\documentclass[10pt]{amsart}
\usepackage{hyperref}
\usepackage[bottom]{footmisc}
\usepackage{geometry}
\usepackage{xcolor}
\usepackage{graphicx}
\usepackage{enumitem}                               
\usepackage{amssymb}
\usepackage{mathrsfs}
\usepackage{amsthm}

\newtheorem{theorem}{Theorem}

\newtheorem{theoremn}{Theorem}[section] 
\newtheorem{lemman}[theoremn]{Lemma}

\newtheorem{corollaryn}[theoremn]{Corollary}

\newtheorem*{theorem*}{Theorem}
\newtheorem*{lemma*}{Lemma}

\theoremstyle{definition}
\newtheorem{definition}[theoremn]{Definition}

\newtheorem{remark}[theoremn]{Remark}
\newtheorem{example}[theoremn]{Example}

\newtheoremstyle{definition*}
{\topsep}
{\topsep}
{}
{0pt}
{\bfseries}
{.}
{ }
{\thmname{#1}\thmnumber{ #2}\thmnote{ (#3)}}

\theoremstyle{definition*}
\newtheorem*{definition*}{Definition}
\newtheorem*{remark*}{Remark}
\newtheorem*{claim*}{Claim}
\newtheorem*{corollary*}{Corollary}

\begin{document}
	\baselineskip 13pt
	
	\title[some special coprime actions]{some special coprime actions\\ and their  consequences}
	\author{G\"{u}l\.{I}n Ercan$^{*}$}
	\address{G\"{u}l\.{ı}n Ercan, Department of Mathematics, Middle East
		Technical University, Ankara, Turkey}
	\email{ercan@metu.edu.tr}
	
	\author{\.{I}sma\.{I}l \c{S}. G\"{u}lo\u{g}lu}
	\address{\.{ı}sma\.{ı}l \c{S}. G\"{u}lo\u{g}lu, Department of Mathematics, Do%
		\u{g}u\c{s} University, Istanbul, Turkey}
	\email{iguloglu@dogus.edu.tr}
	
	\author{M. Yas\.{ı}r Kizmaz}
	\address{M. Yas\.{ı}r Kizmaz, Department of Mathematics, Bilkent University, Ankara, Turkey}
	\email{yasirkizmaz@bilkent.edu.tr}
	
	\author{Danila O. Revin}
	\address{Danila O. Revin, Sobolev Institute of Mathematics SB RAS}
	\email{revin@math.nsc.ru}
	\thanks{$^{*}$Corresponding author}
	\subjclass[2000]{20D10, 20D15, 20D45}
	\keywords{coprime action, $p$-length, Fitting series, Frobenius action, simple group}
	
	\maketitle
	
	\begin{abstract} Let a  group $A$ act on the group $G$ coprimely. Suppose that the order of the fixed point subgroup $C_G(A)$ is not divisible by an arbitrary but fixed prime $p$. In the present paper we determine bounds for the $p$-length of the group $G$ in terms of the order of $A$, and investigate how some $A$-invariant $p$-subgroups are embedded in $G$ under various additional assumptions.
		
	\end{abstract}
	
	\section{introduction}
	
	All groups considered are finite. Let a group $A$ act on the group $G$. The nature of this action has some radical consequences on the structures of both $G$ and $A$ and also leads to some bounds for the invariants of both in terms of the other's. So much research is devoted to studying coprime action, that is the case $(|G|,|A|)=1$ due to the existence of well known nice relations between $G$ and $A$. The present paper is concerned with the consequences of some coprime actions with the additional condition common to all of them such that the order of the fixed point subgroup $C_G(A)$ is not divisible by an arbitrary but fixed prime $p$. In Section 2 we handle the case where $A$ acts with regular orbits, that is, for every proper subgroup $B$ of $A$ and every elementary abelian $B$-invariant section $S$ of $G,$ there exists some $v\in S$ such that $C_B(v)=C_B(S)$; and bound the $p$-length of the group $G$. Namely, we prove
	\vspace{0,2cm}
	
	\begin{theorem}\label{A} Let $A$ be a group acting coprimely and with regular orbits on the solvable group $G$.  Suppose that $B$ is a subgroup of $A$ such that $\bigcap_{a\in A} [G,B]^a=1.$ If $p$ is a prime not dividing $|C_G(A)|$ then $\ell_p(G)\leq \ell(A:B)$.
	\end{theorem}

	Here, $\ell(A:B)$ denotes the length (by the number of inclusions) of the longest chain of subgroups of $A$ that starts with $B$ and ends with $A$. Simply we use $\ell(A)$ instead of $\ell(A:1)$.	The proof of Theorem A, one immediate consequence of which is presented below, involves a series of reductions similar to the techniques used in the proof of Theorem 2.1 of \cite{Tu}.
	\vspace{0,2cm}
	
	\begin{corollary*} Let $A$ be a group acting coprimely and with regular orbits on the solvable group $G$. For any prime $p$ not dividing $|C_G(A)|$ we have $\ell_p(G)\leq \ell(A)$.
	\end{corollary*}
	
	Recall that the coprime action of $A$ guarantees the existence of  $A$-invariant Sylow subgroups. In \cite{Y} Kızmaz  studied the structure of a group $G$ admitting a coprime automorphism $\alpha$ such that $G$ has a unique $\alpha$-invariant Sylow $p$-subgroup for a prime $p$ where $C_G(\alpha)$ is a $p'$-group and also asked about replacing $\alpha$ by any subgroup $A$ of $ AutG$. The rest of the paper is concerned with this kind of extensions. More precisely, we consider the situation that $G$ contains a unique $A$-invariant Sylow $p$-subgroup where $C_G(A)$ is a $p'$-group. It should be noted that $G$ contains a unique $A$-invariant Sylow $p$-subgroup $P$ if and only if $C_G(A)$ normalizes $P.$ A first partial answer in this direction is the second result of Section 2 which bounds the $p$-length of $G$ by  $\ell(A)$ in the case where $A$ is abelian and $G$ is $p$-separable. Namely we have

		\begin{theorem}\label{B} Let $A$ be an abelian group acting coprimely on the $p$-separable group $G$. Suppose that $G$ contains a unique $A$-invariant Sylow $p$-subgroup and that $C_G(A)$ is a $p'$-group. Then $\ell_p(G)\leq \ell(A)$.
		\end{theorem}
	
	Section 3 is devoted to the extensions of some results of \cite{Y} and firstly includes the following answer to its Question 4.1 when $A$ is abelian.

		\begin{theorem}\label{C} Let $A$ be an abelian group acting coprimely  on the solvable group $G$. Suppose that $G$ contains a unique $A$-invariant Sylow $p$-subgroup $P$ for an odd prime $p$ where $C_G(A)$ is a $p'$-group. Then  $P\leq F_{2\ell(A)}(G)$.
		\end{theorem}
	
	The next result extends Proposition 3.1 in \cite{Y} for arbitrary $A$ under some additional assumptions.

		\begin{theorem}\label{D} Let $A$ be a group acting coprimely on the group $G$ and let $P$ be an
		$A$-invariant Sylow $p$-subgroup of $G$ for a prime $p$ dividing the order of $G$. Suppose
		that $p$ does not divide $|C_G(A)|$ and that $C_G(a)$ normalizes $P$ for all $1 \ne a \in A$  of
		prime order. If $|GA|$ is odd then $P\leq F_2(G)$. Furthermore, $P\leq F(G)$ if $A$ is of prime order and $G$
		is solvable.
	\end{theorem}

It should be noted that the assumptions of Theorem D are indispensable because the group $G=PSL(2,2^5)$ admits an automorphism of order $a$ of order $5$ such that $C_G(a)=PSL(2,2)$ normalizes an $\langle a\rangle$-invariant $11$-subgroup of $G$ while $F(G)=F_2(G)=1.$ 
	
	Section 4 mainly includes Theorem E below which yields Theorem C in \cite{Y} as an immediate corollary.

	\begin{theorem}\label{E} Let $A$ be a group acting on a group $G$ coprimely. Suppose that $U$ is an $A$-invariant $p$-subgroup of $G$ such that $C_U(a)=1$ for each $1\neq a\in A$ and that $C_G(A)$ normalizes $U$.  Then $U\leq O_p(G)$ if the following hold:
	\begin{enumerate}
	\item [(i)] $G$ is $PSL(2,2^r)$ free for all $1\neq r$ dividing $|A|$ in case where $p \mid 2^r+1$,
	
 \item [(ii)] $G$ is $Sz(2^r)$ free for all $1\neq r$ dividing $ |A|$ in case where $p \mid 4^r+1$.

\end{enumerate}
		\end{theorem}

	The notation and terminology are standard.	
	
	\section{bounding the $p$-length}
	
	Let $G$ and $A$ be groups where $A$ acts on $G$. The concept of an $A$-tower will be frequently used throughout the paper.
	
	\begin{definition} (Definition $1.1$ and $1.2$ of \cite{T}) We say that a sequence $(S_i),i= 1,\ldots ,t$ of $A$-invariant subgroups of $G$ is an $A$-tower of $G$ of height $t$ if the following are satisfied:
		\begin{itemize}
			\item[(1)] $S_i$ is a $p_i$-group, $p_i$ is a prime, for $i= 1,\ldots ,t$;
			\item[(2)] $S_i$ normalizes $S_j$ for $i\leq j$;
			\item[(3)] Set $P_t=S_t,\,  P_i=S_i/T_i$ where $T_i=C_{S_i}(P_{i+1}), \, i= 1,\ldots ,t-1$ and we assume that $P_i$ is not trivial for $i= 1,\ldots ,t$;
			\item[(4)] $p_i\ne p_{i+1}, \, i= 1,\ldots ,t-1$.
		\end{itemize}
		
		\vspace{1mm}
		
		An $A$-tower $(S_i),\, i=1,\ldots ,t$ of $G$ is said to be irreducible if the following are satisfied:
		\begin{itemize}
			\item[(5)]  $\Phi(\Phi(P_i)) = 1,\,  \Phi(P_i)\leq Z(P_i)$ and, if $p_i\ne 2$, then $P_i$ has exponent $p_i$ for $i= 1,\ldots ,t.$  Moreover $P_{i-1}$ centralizes $\Phi(P_i)$;
			\item[(6)] $P_1$ is elementary abelian;
			\item[(7)] There exists $H_i$, an elementary abelian $A$-invariant subgroup of $P_{i-1}$ such that $[H_i,P_i] = P_i$;
			\item[(8)] $(\prod_{j=1}^{i-1}S_j)A$ acts irreducibly on $P_i/\Phi(P_i)$.
		\end{itemize}
	\end{definition}

	\begin{remark} When the action is coprime $G$ and $G$ is solvable, $G$ contains an $A$-invariant Sylow $p$-subgroup for every prime $p$ dividing $|G|$. This leads to the existence of $A$-towers so that the Fitting height of the group $G$ coincides with the maximum of the heights of all possible $A$-towers in $G$.
	\end{remark}

	The essence of the proof of Theorem A lies in the following
	
	\begin{lemman}\label{lem} Let a group $A$ act on the solvable group $G$ coprimely and let $p$ be a prime dividing $|G|.$ Then $G$ contains an $A$-tower having exactly $\ell_p(G)$-many $p$-terms.
	\end{lemman}
	\begin{proof}

		We proceed by induction on $|G|.$ Suppose that $O_{p'}(G)\ne 1$. Then the group $\bar{G}=G/O_{p'}(G)$ contains an $A$-tower $\bar{S}_1,\ldots ,\bar{S}_t$ having exactly $\ell_p(G)$-many $p$-terms where $\bar{S}_i$ is a subgroup of $\bar{G}$ for each $i=1,\ldots ,t.$ By Lemma 1.6 in \cite{T} there is an $A$-tower $S_1,\ldots ,S_t$ of $G$ which maps to $\bar{S}_i,\; i=1,\ldots ,t$ having exactly $\ell_p(G)$-many $p$-terms. This contradiction shows that $O_{p'}(G)=1$ and hence $F(G)=O_p(G)$.
		
		Similarly an induction argument applied to the action of $\bar{G}=G/O_{p,p'}(G)$ yields an $A$-tower $\bar{S}_1,\ldots ,\bar{S}_t$ of $\bar{G}$ having exactly $(\ell_p(G)-1)$-many $p$-terms. Notice that $\bar{S}_t$ is a subgroup of $O_p(\bar{G})$. By Lemma 1.6 in \cite{T} again, we get an $A$-tower $S_1,\ldots ,S_t$, of $G$ which maps to $\bar{S}_i,\; i=1,\ldots ,t$.  We see that $[S_1,\ldots , S_{t-1}, S_t]\ne 1$
		where $S_t$ is a $p$-group contained in $O_{p,p',p}(G)$. Since $S_t/S_t\cap O_p(G)$ acts faithfully on $O_{p,p'}(G)/O_p(G)$, there exists a $q$-group $Q\leq O_{p,p'}(G)$ for some prime $q\ne p$ such that $[S_1,\ldots , S_{t-1}, S_t, Q]\ne 1.$ We may assume that $Q$ is $A$-invariant. It follows that the sequence $$S_1,\ldots , S_{t-1}, S_t, Q, F(G)$$ forms an $A$-tower having exactly $\ell_p(G)$-many $p$-terms which is a contradiction completing the proof.
	\end{proof}
	
	\begin{proof}[Proof of Theorem \ref{A}] We choose a counterexample with minimum $|GA|+|A:B|$. If $A = B$ then $G=C_G(A)$ whence $G=O_{p'}(G)$ and the result holds. Hence we may assume that $\ell(A:B)\geq 1$. We may also assume that $O_{p'}(G)=1$, that is $F(G)=O_p(G).$
	
	Lemma \ref{lem} guarantees the existence of an $A$-tower  $S_1,\ldots ,S_t$ in $G$ having exactly $\ell_p(G)$-many  $p$-terms where $S_1$ and $S_t$ are both $p$-groups with $t\geq 3.$ We may assume that this tower is irreducible by Lemma 1.4 in \cite{T}.  An induction argument gives that $G=\prod_{i=1}^{t}S_i$ with $T_{t-1}=1.$ Set $H=\prod_{i=1}^{t-1}S_i$ and $  R=P_{t-1}=S_{t-1}$. From now on we shall proceed over a series of steps:\\
	
	\textit{Step 1. For all $C\leq A $ such that $B\leq C$ and $\ell(C:B)\geq 1$ we have $R=[R,C]^{H}$.}\\
	
	Assume the contrary, that there exists $C\leq A $ such that $B\leq C$ and $\ell(C:B)\geq 1$ so that $R\ne [R,C]^{H}$. Set $R_0=[R,C]^{H}\Phi(R)$. Since $R/\Phi(R)$ is irreducible as an $HA$-module, we have $\bigcap_{a\in A}{R_0}^a=\Phi(R)$ and $\bigcap_{a\in A}{C_G(R/R_0)}^a=C_G(R/\Phi(R))$.
	
	Set $\overline{H}= H/C_H(R/\Phi(R))$. As $[\overline{S}_1,\ldots ,\overline{S}_{t-2}]=\overline{S}_{t-2}\ne  1$, the sequence  $\overline{S}_1,\ldots ,\overline{S}_{t-2}$ is an $A$-tower of $\overline{H}$ having exactly $(\ell_p(G)-1)$-many  $p$-terms. This forces that $$\ell_p(\overline{H})=\ell_p(G)-1.$$ Notice that $[R/R_0,C]=1$ and $\overline{R_0}\lhd \overline{H}$. Then, by the three subgroups lemma, $[H,C]\leq C_H(R/R_0)$ and hence $$\bigcap_{a\in A}[\overline{H},C]^a=\overline{\bigcap_{a\in A}[H,C]^a}\leq \overline{\bigcap_{a\in A}{C_H(R/R_0)}^a}=1.$$ Now an induction argument applied to the action of $A$ on $\overline{H}$ implies $$\ell_p(G)-1=\ell_p(\overline{H})\leq \ell(A:C)\leq \ell(A:B)-1.$$ This forces that  $\ell_p(G)\leq \ell(A:B)$, which is a contradiction establishing the claim.\\

	\textit{Step 2. Recall that $P_{t-2}=S_{t-2}/T_{t-2}$. Set $Q=S_{t-2}$ and $K=\prod_{i=1}^{t-2}S_i$. Then for all $D\leq A $ such that $B\leq D$ and $\ell(D:B)\geq 2$ we have $Q=[Q,D]^{K}\Phi$ where $\Phi=\Phi(Q)T_{t-2}$.}\\
	
	Assume the contrary, that there exists $D\leq A $ such that $B\leq D$ and $\ell(D:B)\geq 2$ so that $Q\ne [Q,D]^{K}\Phi$. Let $Q_0=[Q,D]^{K}\Phi$. Since $Q/\Phi$ is an irreducible $KA$-module, $\bigcap_{a\in A}{Q_0}^a=\Phi$ and $\bigcap_{a\in A}(C_G(Q/Q_0)^a=C_G(Q/\Phi)$.
	
	Set $\overline{K}= K/C_K(Q/\Phi)$. As $[\overline{S}_1,\ldots ,\overline{S}_{t-3}]=\overline{S}_{t-3}\ne  1$, the sequence  $\overline{S}_1,\ldots ,\overline{S}_{t-3}$ forms an $A$-tower of $\overline{K}$. It follows that
	$$\ell_p(G)-2\leq \ell_p(\overline{K})\leq \ell_p(G)-1.$$ Then we also have $\ell_p(K/C_K(Q/Q_0))\geq \ell_p(G)-2.$ Since $D$ acts
	trivially on $Q/Q_0$, $[K,D]\leq C_K(Q/Q_0)$. Therefore we have $$\bigcap_{a\in A}[\overline{K},D]^a=\overline{\bigcap_{a\in A}[K,D]^a}\leq \overline{\bigcap_{a\in A}{C_K(Q/Q_0)}^a}=1.$$ By induction applied to the action of $A$ on $\overline{K}$ we get $$\ell_p(G)-2\leq \ell_p(\overline{K})\leq \ell(A:D)\leq \ell(A:B)-2$$ which forces that  $\ell_p(G)\leq \ell(A:B)$. This contradiction establishes the claim.\\
	
	\textit{Step 3. Final contradiction.}\\
	
	Since the $A$-tower $S_1,\ldots ,S_t$ is irreducible, the groups $P$ and $Q/C_Q(P)$ are special. Furthermore they are of exponent $p_{t-1}$(resp. $p_{t-2}$) if $p_{t-1}$ and $p_{t-2}$ are odd. We are now ready to apply \cite[Theorem 1.1]{Tu},  to the action of $S_{t-1}S_{t-2}A$ on the Frattini factor group of $S_t$ and get $C_{S_t}(A)\ne 1.$ This contradiction completes the proof.
	\end{proof}

\begin{proof}[Proof of Theorem \ref{B}] Let $GA$ be a minimal counterexample to the theorem. We may assume that $O_{p'}(G)=1$. Let $O_i(G),\; i=1,\ldots ,\ell(G)$ be defined for $\pi=\{p\}$ as in \cite{Ku} where $\ell(G)$ is the least positive integer such that $G=O_{\ell(G)}(G).$ If $P\leq O_{\ell(G)-1}(G),$ an induction argument applied to the action of $A$ on the group $O_{\ell(G)-1}(G)$ implies that $\ell_p(G)\leq \ell(A)$, which is not the case. Therefore we may assume that $G=O_{\ell(G)-1}(G)P.$ By \cite[Lemma 4.3]{Ku} there exists a sequence $A_1,\ldots , A_{\ell(G)}$ of $A$-invariant sections of $G$ satisfying the conditions $(1.10.a)-(1.10.f)$ of \cite{Ku}. Furthermore, as a consequence of \cite[Lemma 4.3 (a)]{Ku}, the following are satisfied:\\
	
	$(a)$ $A_i$ is a $p$-group (or a $p'$-group), respectively $A_{i+1}$ is a $p'$-group (or a $p$-group), and $A_{\ell(G)}\leq G$. In our case we see that $A_1$ and $A_{\ell(G)}$ are both $p$-groups. In particular $C_{A_1}(A)=1$\\
	
	$(b)$ $\ell_p(G)$ is equal to the number of $p$-groups among the sections $A_i$, for  $i=1,\ldots ,\ell(G).$\\

	$(c)$ $ [A_i, A_{i-1}] = A_i$, for $i = 2 , \ldots ,\ell(G)$.\\
	
	More precisely, the sequence $A_1,\ldots , A_{\ell(G)}$ is an $A$-tower. Since $A$ acts fixed point freely on $A_1$ there is a nonidentity element $a\in A$ of prime order such that $[A_1,a]\ne 1$. It follows by Theorem 3.1 in \cite{T} that there is a sequence of $A$-invariant subgroups $C_2,\ldots ,C_{\ell(G)}$ each of which is centralized by $a$ so that it forms an $A$-tower. This forces that the $C_{O_{\ell(G)}(G)}(a)$ has $p$-length $\ell_p(G)-1$. We then apply induction to the action of $A/\langle a\rangle $ on $C_{O_{\ell(G)}(G)}(a)$ and get $\ell_p(G)-1\leq \ell-1.$ This contradiction completes the proof.
	\end{proof}

	\section{embedding of the unique $A$-invariant Sylow $p$-subgroup }

	\begin{proof}[Proof of Theorem \ref{C}]  We proceed by induction on $|GA|.$ Let $k$ be the largest such that $p$ divides the order of $F_k(G)/F_{k-1}(G)$. Assume that $k>2\ell(A).$ Then there is an $A$-tower $S_1, S_2,\ldots , S_{2\ell(A)+1} $ of $G$ where $S_1$ is a $p$-group. We may assume that this tower is irreducible. Set $V=P_2/\Phi(P_2).$ We have $C_V(A)=1$ as $[C_{P_2}(A),P_1]=1$. If the group $P_1A$ is Frobenius we would have $C_V(A)\ne 1$, which is a contradiction. Thus there exists $1\ne a\in A$ such that $C_{P_1}(a)\ne 1$. By \cite[Theorem 3.1]{T}, we see that $$[C_{S_1}(a),\ldots, C_{S_{j-1}}(a),  C_{S_{j+1}}(a), \ldots ,  C_{S_{2\ell(A)+1}}(a)]\ne 1.$$ Indeed we have one of the two cases:\\
	
	(1) The group $X=C_{S_1}(a) \ldots C_{S_{j-1}}(a)C_{S_{j+1}}(a) \ldots  C_{S_{2\ell(A)+1}}(a)$  is of Fitting height $2\ell$;\\
	
	(2) The group $Y=C_{S_1}(a) \ldots C_{S_{j-1}}(a)C_{S_{j+2}}(a) \ldots  C_{S_{2\ell(A)+1}}(a)$  is of Fitting height $2\ell-1$.\\
	
	In case (1) we apply induction to the action of $A/\langle a\rangle$ on $X$ and get $C_{S_1}(a)\leq F_{2\ell-2}(X)$, and in case (2) we apply induction to the action of $A/\langle a\rangle$ on $Y$ and get $C_{S_1}(a)\leq F_{2\ell-2}(Y)$, which are both impossible. Hence the claim is established.
	
\end{proof}

	\begin{proof}[Proof of Theorem \ref{D}] Suppose that $k$ is the largest such that $p$ divides the order of $F_k(G)/F_{k-1}(G)$. Assume that $k>2.$ Then there is an $A$-tower $S_1, S_2, S_3$ of $G$ where $S_1$ is a $p$-group. We may assume that this tower is irreducible. Then $C_{P_2} (a) \leq \Phi(P_2)$
	for all $1 \ne a \in A$
	as $[C_{P_2} (a), P_1] = 1$. It follows that $P_2 = [P_2, a]$ for all $1 \ne a \in A$, that is, the group
	$P_2A$ is Frobenius-like with kernel $P_2$. By \cite[Corollary C]{GE}, applied
	to the action of $P_2A$ on $P_3$ we observe that $C_{P_3} (A)\ne  1$ which forces that $p_3\ne p$.
	
	If the group $P_1A$ were Frobenius then we would have $C_{P_2} (A)\not \leq \Phi(P_2)$ which is not the case.
	Thus there exists $1 \ne b \in A$ such that $C_{P_1} (b) \ne 1$. We may assume that $b$ is of prime
	order. It follows that $[C_{S_1}
	(b), P_2, P_3]\ne 1$ as $P_1$ acts faithfully on $P_2.$ Since $p$ is odd,
	by \cite[Theorem 3.1]{T},  we get either $[C_{S_1}
	(b), C_{S_2}
	(b)] \ne 1$ or $[C_{S_1}
	(b), C_{S_3}
	(b)]\ne 1.$
	Both are impossible by hypothesis and the first claim follows.
	
	Finally assume that $A$ is of prime order and $G$
	is solvable. In this case let $S_i$ and $S_{i-1}$ be two successive terms of an $A$-tower such that $S_{i-1}$ is a $p$-group. Notice that $C_{S_{i-1}}(A)$ and $C_{P_i/\Phi(P_i)}(A)$ are both trivial. This forces by Thompson's celebrated theorem that $S_{i-1}$ centralizes $P_i$ which is a contradiction. Hence the proof is complete.
	\end{proof}
	
	\begin{example}

		Let $H=P\times A$ where $P\lhd H$ is cyclic of order $7$ and $A$ is cyclic of order $9.$ Suppose that $B=\Omega_1(A)\lhd H$ and $H/B$ is a Frobenius group of order $21.$ Then $Soc(H)=P\times B$ and is cyclic of order $21$.  By Theorem 10.3 on page 173 in \cite{Do} there exists an elementary abelian $5$-group $V$ which is a faithful and irreducible $H$
		-module. Consider the semidirect product $VH$ and let $G$ be the subgroup $VP$. Then $A$ acts coprimely on $G$, $[V,P]=V$, and $[V,B]=V$ whence the group $VA$ is Frobenius. Now $C_G(A)=1$ and $C_G(a)=P$ for any $1\ne a\in A$ of prime order. This example shows that $P$ is not necessarily contained in $F(G)$ under the hypothesis of Theorem \ref{D}.

	\end{example}

	\section{embedding of some $A$-invariant $p$-subgroups within the group}
	
	Although this section is devoted to a proof of Theorem E we want first to emphasize a special case of this result due to the simplicity of its proof.
	
	\begin{theoremn} \label{separab} Let $A$ be a group acting on the $p$-separable group $G$ coprimely and let $U$ be an $A$-invariant $p$-subgroup of $G$ such that $C_U(a) = 1$ for each $1 \ne a\in A$. If
		$C_G(A)$ normalizes $U$, then $U\leq O_p(G)$.
		
	\end{theoremn}
	
	\begin{proof} Let $G$ be a minimal counterexample to the theorem. We can easily observe
		that by an induction argument applied to the action of $A$ on $G/O_p(G)$ we get
		$O_p(G) = 1$. Another induction argument applied to the action of $A$ on $O_{p'} (G)U$
		yields that $G = O_{p'} (G)U$. By hypothesis, the group $UA$ is Frobenius with kernel $U$. Let $Q$ be
		a $UA$-invariant Sylow $q$-subgroup of $O_{p'} (G)$ on which $U$ is nontrivial. Set $V$ be
		the Frattini factor group of $Q$. W.l.o.g. we may assume that $V$ is irreducible as a
		$UA$-module. It is well known that $C_V (A)\ne 1$. On the other hand $[C_Q(A),U] = 1$
		and so $C_Q(A) \leq \Phi(Q)$, that is $C_V (A) = 1$, which is a contradiction.
	\end{proof}

	We now prove some lemmas which will be used in the proof of Theorem E.
	
	\begin{lemman}\label{centralizer}
		Let $G=PGL(2,p^r)$ for some positive integer $r$ and let $P$ be a Sylow $p$-subgroup of $G$. Then $C_G(x)\leq P$ for any nonidentity $x\in P$.
	\end{lemman}
	\begin{proof}

		Let $\Gamma=GL(2,p^r)$ and $F$ be a field of order $p^r$. Let $A =  \begin{bmatrix} 1&t\\0&1 \end{bmatrix}$ for some $t\in F^*$ and pick $B=   \begin{bmatrix} a&b\\c&d \end{bmatrix}\in \Gamma$ such that $AB=BA$. It follows easily  that $a=d$ and $c=0$, that is,
		
		$$B= \begin{bmatrix} a&0\\0&a \end{bmatrix} \begin{bmatrix} 1&b/a\\0&1 \end{bmatrix}\in Z(\Gamma)Q$$
		where $Q$ is the Sylow $p$-subgroup of $\Gamma$ which consists of upper triangular matrices. Thus, we obtain $C_\Gamma(A)\leq Z(\Gamma)Q$. Write $\overline \Gamma =\Gamma/Z(\Gamma)=G$. Since $Z(\Gamma)$ is a $p'$-subgroup of $\Gamma$, we get  $C_G(\overline A)=C_{\overline \Gamma}(\overline A)=\overline{C_\Gamma(A)}\leq \overline Q$ by \cite[Lemma 7.7]{1}. Note that $\overline Q$ is a Sylow $p$-subgroup of $G$, and so by taking an appropriate conjugate we obtain that $C_G(x)\leq P$ for any $1\neq x\in P$.
	\end{proof}
	
	\begin{lemman}\label{elimination of PSL(2,3)}
		Let $G\cong PSL(2,3^r)$ and $H\cong PSL(2,3)$ be a subgroup of $G$. Suppose that $U$ is a subgroup of $G$ which is normalized by $H$ and which has trivial intersection with $H$. Then $U=1$.
	\end{lemman}
	
	\begin{proof}
		Let $G$ be a minimal counter example to the lemma. Note that the order of $G$ is $q(q-1)(q+1)/2$ where $q=3^r$. Let now $M$ be a maximal subgroup of $G$ that contains the subgroup $HU$. The possible structure of $M$ is given in \cite[Corollary 2.2]{2}(a)-(h). Since $\pi(H)=\{2,3\}\subseteq \pi(M)$ and $q=3^r$ for $r\geq 2$, the group $M$ can be only one of the groups described in $(c),(d),(e)$ or $(h)$ of \cite[Corollary 2.2]{2}. We shall complete the proof by obtaining a separate contradiction for each case below.
		
		Suppose that \cite[Corollary 2.2 (c)]{2} holds. Then $M$ is of order $q(q-1)/2$ as $q$ is odd, and so is the normalizer of a Sylow $3$-subgroup. It follows that $M$ is $3$-closed and hence $H\cong A_4$ is also $3$-closed, which is not the case.
		
		Suppose that \cite[Corollary 2.2 (d)]{2} holds, that is, $M\cong PSL(2,3^{r_0})$ for some $r_0<r$. It follows that $U=1$  by the minimality of $G$, which is a contradiction.

		Suppose that \cite[Corollary 2.2 (e)]{2} holds. Then $M\cong PGL(2,3^{r_0})$ for some $r_0<r$. Set $q_0=3^{r_0}$ and let $M_0$ be the subgroup of $M$ isomorphic to $PSL(2,3^{r_0})$. Note that the order of $M$ is $q_0(q_0^2-1)$, and so $|M:M_0|=2$. Clearly $H\leq M_0$ and $H$ normalizes $U\cap M_0$ as $M_0\lhd M.$ By the minimality of $G$, we see that $U\cap M_0=1$, and so $|U|=2$ as $|U|>1.$ It follows then that $[H,U]=1$ as $H$ normalizes $U$. Let now $h\in H$ be of order $3$. Then $U\leq C_M(h)\leq P$ where $P$ is a Sylow $3$-subgroup of $M$ by Lemma \ref{centralizer}, which leads to a contradiction.
		
		Finally suppose that \cite[Corollary 2.2 (h)]{2} holds. Then $M\cong A_5$ and so it is apparent that the subgroup $H\cong A_4$ does not normalize any nontrivial such $U$.

	\end{proof}

	\begin{lemman}\label{3 cases}
		Let $G=L(q^f)$ be a simple group of Lie type over the field $F$ of order $q$ where $q$ is a power of a prime $u$, and let $a$ be the  automorphism of $G$ induced by the field automorphism $x\mapsto x^q$. Suppose that $O^{u'}(C_G(a))=L(q)$ is not isomorphic to one of $PSL(2,2)$, $PSL(2,3)$ and $Sz(2)$. Then there is no nontrivial subgroup $K$ of $G$ which is normalized by $C_G(a)$ and which has trivial intersection with $C_G(a)$.
	\end{lemman}	
	\begin{proof} 	Let $K$ be a subgroup of $G$ which is normalized by $C_G(a)$ and which has trivial intersection with $C_G(a)$. By \cite[Theorem 1]{Bu}, there exists a positive integer $r$ dividing $f$ such that $$L(q^r)\cong O^{u'}(C_G(a^r))\le C_G(a)K\le C_G(a^r)^*$$ where $C_G(a^r)^*$ is generated by $Inn (C_G(a^r))$ and some diagonal automorphisms of the simple group $O^{u'}(C_G(a^r))$. Note that $r<f$ because otherwise $G=C_G(a)K$ which implies by the  simplicity of $G$ that $K=1.$ If $r=1$ then $C_G(a)K\leq Aut (C_G(a))$ and so $K$ normalizes $C_G(a)$, that is, $K=1$ as desired. Then by induction applied to the action of $\langle a\rangle/\langle a^r\rangle$ on $C_G(a^r)$ with $C_G(a^r)\cap K$ we obtain $C_G(a^r)\cap K=1.$ It follows that $C_G(a^r)=C_G(a)$ and so $[C_G(a^r),K]\le C_G(a^r)\cap K=1.$ Due to the faithful action of $K$ on $C_G(a)$ we get $K=1.$
	\end{proof}

	\begin{lemman}\label{lie type}
		Let $A$ be a nontrivial automorphism  group  of a nonabelian simple group $G$  and $(|A|,|G|)=1$. Then $G$ is a simple group of Lie type and $A$ is cyclic.
	\end{lemman}
	
	\begin{proof} It is clear that $A$ is isomorphic to a subgroup of $Out(G)=Aut(G)/Inn(G)$. According to the classification of finite simple groups \cite[Theorem~0.1.1]{AschLyoSmSol} and the well-known information on outer automorphism groups (see \cite[Theorem~5.2.1 and Tables 5.3a--5.3.z]{D}), one of the following statements holds:
		\begin{itemize}
			\item $G$ is isomorphic to an alternating group of degree $\geq 5$ and $|Out(G)|\in\{2,4\}$;
			\item $G$ is isomorphic to one of 26 sporadic groups and $|Out(G)|\in\{1,2\}$;
			\item $G$ is isomorphic to a simple group of Lie type.
		\end{itemize}
		Since $|G|$ is even (by the Feit--Thompson theorem \cite{FT}, for example), $G$ must be isomorphic to a simple group of Lie type. Then  $Out(G)$ is solvable. More exactly, according to \cite[Theorem~2.5.12]{D} $Out(G)$ has a normal subgroup $Outdiag(G)$ such that $\pi(Outdiag(G))\subseteq\pi(G)$. Moreover, $\overline{O}=Out(G)/Outdiag(G)$ has a normal cyclic subgroup $\Phi$ (isomorphic to the automorphism group of the ground field) such that $\overline{O}/\Phi\in\{1,2\}$ or $G\cong D_4(q)$ and $\overline{O}/\Phi\cong S_3$. In all cases $\pi(\overline{O}/\Phi)\subseteq\pi(G)$. This implies that $A$ is isomorphic to a subgroup of $\Phi$. In particular, $A$ is cyclic.
	\end{proof}
	\begin{lemman}\label{Suzuki}
		Let $G=Sz(q)$, where $q=2^r$ for some odd $r$. Then every solvable subgroup  $H$ of $G$ such that $5\in\pi(H)$ is contained in a subgroup $M$ of $G$ with the following properties.
		\begin{itemize}
			\item $|M|=4(q-\varepsilon\sqrt{2q}+1)$, where $\varepsilon\in\{+,-\}$ is uniquely defined by the relation $5\mid (q-\varepsilon\sqrt{2q}+1)$;
			\item $M$ is the Frobenius group with  cyclic kernel of order $q-\varepsilon\sqrt{2q}+1$ and cyclic complement of order $4$.
		\end{itemize}
		In particular, if  $U$ is a nontrivial $p$-subgroup of $G$ such that $5$ divides $|N_G(U)|$ then $U$ is cyclic and $|U|$ divides $q-\varepsilon\sqrt{2q}+1$.
	\end{lemman}
	
	\begin{proof} According to \cite[Theorem~9]{Suz}, every proper subgroup of $G$ is conjugate to a subgroup of one of the following subgroups:
		\begin{itemize}
			\item Frobenius group of order $q^2(q-1)$;
			\item dihedral group $B_0$ of order $2(q-1)$;
			\item the normalizer $B_1$ of a cyclic group $A_1$ of order $q-\sqrt{2q}+1$, $|B_1|=4|A_1|$;
			\item the normalizer $B_2$ of a cyclic group $A_2$ of order $q+\sqrt{2q}+1$, $|B_2|=4|A_2|$;
			\item $Sz(2^{r_0})$ where $r_0$ divides $r$.
		\end{itemize}
		It is easy to see that $5$ divides only the orders of $Sz(2^{r_0})$ and exactly one of $B_1$ and~$B_2$. We choose $i\in\{1,2\}$ such that $5$ divides $|B_i|$.  It follows by induction from this remark that $H$ contains a normal cyclic 2-complement $T$ and $T$ contains a cyclic subgroup of order $5$. Since $A_i$ contains a Sylow 5-subgroup of $G$, we may assume that $A_i\cap T\ne 1$. By \cite[Proposition~16]{Suz}, the centralizer of every nontrivial element of $A_i$ coincides with $A_i$. Therefore, $T\leq C_G(A_i\cap T)\leq A_i$ and $H$ normalizes $C_G(T)=A_i$. This means that $H\leq B_i$. Moreover, $B_i/A_i=B_i/C_{B_i}(Z)$ is isomorphic to a subgroup of the automorphism group of the subgroup $Z\leq A_i$ of order~$5$. Consequently, $B_i/A_i$ is cyclic. Since $C_{B_i}(a)=A_i$ for every nontrivial $a\in A_i$, $B_i$ is a Frobenius group with the cyclic kernel $A_i$ and a cyclic complement of order $4$.
	\end{proof}

	We need an extension of  \cite[Lemma 3.3]{Y} in the proof of Theorem E.
	\begin{lemman}\label{classification}
		Let $G$ be a nonabelian simple group and let $\alpha$ be a coprime automorphism of $G$ of order $r$. Let $U$ be a nontrivial $\alpha$-invariant $p$-subgroup of $G$ such that $C_U(\alpha)=1$. If $C_G(\alpha)$ normalizes $U$, then one of the following holds:
		
		\begin{enumerate}
			\item[a)] $G=PSL(2,2^r)$ and  $r\geq 5$. Moreover, $p\geq 5$ and  $p$ is a divisor of $2^r+
			1$.
			
			\item[b)] $G=Sz(2^r)$ and $r\geq 7$. Moreover, $p\geq 5$ and  $p$ is a divisor of $2^r\pm \sqrt{2^{r+1}}+1$ where the sign $\pm$ is chosen such that $5$ divides $2^r\pm \sqrt{2^{r+1}}+1$.
		\end{enumerate}
	\end{lemman}
	
	\begin{proof}
		We see that $G$ is a simple group of Lie type by Lemma \ref{lie type}. It follows that $G=PSL(2,2^r)$ or $Sz(2^r)$ by Lemmas \ref{3 cases} and \ref{elimination of PSL(2,3)}. Set $C=C_G(\alpha)$.
		
		Let $G=PSL(2,2^r)$ and set $q=2^r$.  Then we see that $C=PSL(2,2)\cong S_3$. Suppose first that $p$ is odd. A Sylow $p$-subgroup $P$ of $G$ is cyclic by \cite[Theorem 8.6.9]{Kurz}. If $p=3$, then $U\cap C$ contains an element of order $3$, which is impossible by the hypothesis. Thus $p\geq 5$. We also have $r\geq 5$ as $r$ is coprime to $|C|=6$. We see that $C$ is contained in a maximal subgroup $D$, which is a dihedral group of order $2(q\pm 1)$ (see \cite[Corollary 2.2 (f) and (i)]{2}). Since $r$ is odd, $3$ is coprime to $q-1=2^r-1$, and so $|D|=2(q+1)=2^r+1$.
		
		Let $T$ be the subgroup of $C$ of order $3$. Clearly, $T$ is normalized by $D$, and so $D=N_G(T) $ as $G$ is simple and $D$ is a maximal subgroup of $G$. Now we claim that $p \mid q+1$. Since  $U$ is cyclic, $Aut(U)$ is abelian. It follows that $C/C_C(U)$ is abelian.  We get that $T\leq C_C(U) $ as $T=C'$, and so $U\leq C_G(T)\leq N_G(T)=D$. As $p$ is odd and $|D|=2(q+ 1)$, we have that $p$ divides $q+1$.  Consequently, we observe that if such an $\alpha$-invariant $p$-subgroup $U$ of $G$ exists, it must be contained in $D=N_G(T)$. On the other hand, $D=N_G(T)$ is $\alpha$-invariant  and $\pi(D)\neq \{2,3\}$ as $r\geq 5$. Pick an $\alpha$-invariant Sylow $p$-subgroup $P$ of $D$ for $p\geq 5$.  Clearly, $P$ is normalized by $C$ and $C_{P}(\alpha)=1$, which completes the proof for this case.

		Assume now that $p=2$ and take a  Sylow $p$-subgroup $P$ of $G$ such that $U\leq P$. In this case, $P$ is elementary abelian of order $2^r$ and $|N_G(P)|=2^r(2^r-1)$ (see ~\cite[Table~1]{Bray}). Since $r$ is odd, we have $(|N_G(P)|,3)=1$. 
		It follows from Lemma~\ref{centralizer} that $C_G(U)=P$ which means that $P$ is a normal subgroup of $N_G(U)$. Therefore, $N_G(U)\leq N_G(P)$ and $(|N_G(U)|,3)=1$, which contradicts the fact that $ N_G(U)\geq C\cong S_3$. Thus, the case $p=2$ is impossible.
		
		Next let $G=Sz(q)$ where $q=2^r$ and $r$ is odd. Note that $C\cong Sz(2)$ which is a Frobenius group of order $20$. Denote $T=O_5(C)$.  Then $T\leq C\leq N_G(U)$. In particular, $|N_G(U)|$ is divisible by $5$. By Lemma~\ref{Suzuki} $U$ must be cyclic and $|U|$ must divide $q-\varepsilon\sqrt{2q}+1$, where $\varepsilon=\pm$ and 5 divides $q-\varepsilon\sqrt{2q}+1$.
		
		If $r=3$ then $q-\sqrt{2q}+1=5$. Consequently, $\varepsilon=+$ and $p=|U|=5$. But in this case $UT$ is contained in a Frobenius subgroup with a cyclic kernel and a cyclic complement of order $4$ by Lemma~\ref{Suzuki}. This means that $U=T$ is contained in $C$, a contradiction. Hence, $r>3$.
		
		Since $5$ divides $|G|$ and $|\alpha|=r$, we have $r>5$. Now the desired statement follows from Lemma~\ref{Suzuki}.

	\end{proof}

	\begin{proof}[Proof of Theorem \ref{E}] Let $G$ be a minimal counterexample to the theorem and choose $U$ of minimal possible order. Then $U>1$. It can be easily observed by an induction argument applied to the action of $A$ on $G/O_p(G)$ that $O_p(G)=1$. Let $N$ be a minimal normal $A$-invariant subgroup of $G$. We shall separate the proof into two cases:
	
	\textit{Case 1.}  Assume that $N=G.$ Then $G$ is characteristically simple, that is, $G=G_1 \times \cdots \times G_n$ where $G_i$ are isomorphic nonabelian simple groups and $A$ acts transitively on $\{G_i\; :\; i = 1, \ldots, n\}$. Let now  $B=N_A(G_1)$ and let $X=G_2\times G_3 \times \ldots G_n$. Note that $X$ is a $B$-invariant normal subgroup of $G$. Assume that $X>1$ and set $\overline G=G/X$. Let $A=\cup_{i=1}^n Ba_i$ be a coset decomposition of $A$  with respect to $B$ where $a_1=1$.  We observe that $C_G(A)=\{\; \prod_{i=1}^ng^{a_i}\; :\; g\in C_{G_1}(B)\}$, and hence $$\overline {C_G(A)}=\overline{C_{ G_1}(B)}={C_{\overline G_1 }(B)}=C_{\overline G}(B).$$ Then $C_{\overline G}(B)$ normalizes $\overline U$. Since $\overline U B$ is a Frobenius group, an induction argument applied to the action of $B$ on $\overline G$ yields that $\overline U \leq O_p(\overline G)=1$, that is, $U\leq X$. It follows that $U=1$ as $A$ acts transitively on $\{G_i\; :\; i = 1, \ldots, n\}$. By this contradiction, we get $X=1$, that is, $G$ is simple. We observe by Lemma \ref{lie type} that $A$ is cyclic. Then, appealing to Lemma \ref{classification}, we obtain the final contradiction in Case 1.

	\textit{Case 2.} Assume that $N<G$. By induction applied to the action of $A$ on $N$, it holds that $U\cap N\leq O_p(N)\leq O_p(G)=1.$ Write $\overline G=G/N$. Then by induction applied to the action of $A$ on $\overline G$, we get $1<\overline U \leq O_p(\overline G)$. Let $H/N=O_p(\overline G)$. Assume $H<G$. Clearly, $U\subseteq H$, and so $U\leq O_p(H)\leq O_p(G)$ by induction applied to $H$. This contradiction shows that $H=G$. Since $\overline G$ is a $p$-group, $NU$ is subnormal in $G$. Thus, if $NU<G$, then we get $U\leq O_p(G)$, which is not the case. Thus, $G=NU$. We get $\Phi(U)\leq O_p(G)=1$ by the minimality of $U$, and so $U$ is an elementary abelian $p$-group. Now $N$ is characteristically simple, that is, $N=N_1\times \cdots \times N_k$ where $N_i,\, i=1,\ldots ,k,$ are simple. Notice that $N$ is nonabelian because otherwise $G$ is $p$-separable and the result follows by Theorem \ref{separab}.
	
	Let $\Omega$ denote the set of $N_i, i=1,\ldots ,k$. Then $UA$ acts transitively on $\Omega.$ Let $\Omega _{1}$ be the $U$-orbit on $\Omega$ containing $N_1$, and set $A_{1}=Stab_{A}(\Omega _{1}).$
	
	Suppose first that $A_{1}=1.$  Clearly, we have $Stab_{A}(N_1)\leq A_{1}=1$. Consider the group $X=\prod_{a\in A}{N_1}^{a}$. Then $
	C_{X}(A)=\{ \prod_{a\in A}n^{a}\; :\;  n\in N_1\}.$ Since $U$ centralizes  $C_{X}(A)$,  $X$ is $UA$-invariant and hence $X=N$ by the minimality of $N.$ That is, $k=|A|$ and so there is a $U$- orbit of length $1$ because otherwise we would have $p$ divides $|A|$. Suppose that $U$ normalizes $N_1$. Then $U$ normalizes $N_i$ for each $i$. This forces that $[N_i,U]=1$ for each $i$ as $[C_N(A),U]=1$, and so $[N,U]=1.$ This contradiction shows that $A_{1}\neq 1.$
	
	Let now $S=Stab_{UA_{1}}(N_1)$ and $U_{1}=U\cap S$. Then $\left\vert
	U:U_{1}\right\vert =\left\vert \Omega _{1}\right\vert =\left\vert
	UA_{1}:S\right\vert .$ Notice next that $(\left| S:U_{1}\right| , \left\vert
	U_{1}\right\vert )=1$\, as \,$(\left\vert U\right\vert ,\left\vert
	A_{1}\right\vert )=1$. Let $S_{1}$ be a complement of $U_{1}$ in $S.$ Then we have $
	\left\vert U:U_{1}\right\vert =\left\vert
	U\right\vert \left\vert A_{1}\right\vert /\left\vert U_{1}\right\vert
	\left\vert S_{1}\right\vert $ which implies that $\left\vert A_{1}\right\vert
	=\left\vert S_{1}\right\vert .$ Therefore we may assume
	that $S=U_{1}A_{1},$ that is, $N_1$ is $A_1$-invariant.

	Let $x\in U$ and $1\neq a\in A_{1}$ such that $({N_1}^{x})^{a}={N_1}^{x}$ holds. Then
	$[a,x^{-1}]\in
	U_{1} $ and so $U_{1}x=U_{1}x^{a}=(U_{1}x)^{a}$ implying the existence
	of an element $g\in U_{1}x\cap C_{U}(a).$ Hence $x\in U_1$. It follows that $Stab_{A_1}({N_1}^x)=1$ for every $x\in U\setminus U_{1}.$ More precisely we have shown that $A_{1}$ is a nontrivial subgroup of $A$ stabilizing exactly one element, namely $N_1$, and all
	the remaining orbits of $A_{1}$ are of length $\left\vert
	A_{1}\right\vert .$
	
	The group $A$ acts transitively on
	$\left\{ \Omega _{i}\; :\; i=1,2,\ldots ,s\right\},$ the collection of $U$-orbits on $\Omega$. Let now  $%
	M_{i}=\prod_{M\in \Omega _{i}}M$ \, for $i=1,2,\ldots ,s.$ Suppose that $s>1.$  Then\\ $%
	A_{1}=Stab_{A}(\Omega _{1})$ is a proper subgroup of $A.$ Let $A=\bigcup_{i=1}^{m}{A_1}a_i$ be the coset decomposition of $A$ with respect to $A_1$. Notice that $C_N(A)=\{ \prod_{i=1 }^m{n^{a_i}}\; :\; n\in C_{M_1}(A_1)\}.$ Since $[C_N(A),U]=1$, we have  $[C_{M_1}(A_1), U]=1.$
	Applying
	induction to the action of $A_{1}$ on $M_1U$ we obtain $U=O_p(M_1U)$, that is $[M_1,U]=1$. Then $[M_i,U]=1$ for each $i$, which is impossible. Thus $A_1=A$ and $\Omega=\Omega_1,$ that is, $U$ acts transitively on $\Omega $.

	It follows that $N_1$ is $A$-invariant. Let $Y=\prod_{a\in A}{N_2}^a$. Since $[C_Y(A),U]=1$, we see that $Y$ is $UA$-invariant which is impossible by the minimality of $N$. Therefore we may assume that $N$ is simple. Moreover, $C_G(N)\cap N=Z(N)=1$. Consequently, $C_G(N)$ is isomorphic to a subgroup of $G/N\cong U$. Therefore, $C_G(N)$ is a normal $p$-subgroup of $G$ and $O_p(G)=1$ implies $C_G(N)=1$. It follows from the three subgroups lemma and the equalities $$[C_{A}(N), N,G]=1\text{ and }[N,G, C_{A}(N)]=[N,C_{A}(N)]=1$$ that $$[G,C_{A}(N),N]=1, \quad [G,C_{A}(N)]\leq C_G(N)=1\quad\text{and}\quad C_{A}(N)= C_{A}(G)=1.$$ This means that  $G$ and $A$ are isomorphically embedded in $Aut(N)$. Moreover, the kernel $C_{GA}(N)$ of the natural homomorphism $GA\rightarrow Aut(N)$ is also trivial because $|G|$ and $|A|$ are coprime. Thus we may consider $GA$ as a subgroup of $Aut(N)$.
	
	Note that $N$ must be isomorphic to a group of Lie type by Lemma \ref{lie type} as it admits a coprime automorphism. We need now some information about the automorphism groups of the simple groups of Lie type given in \cite[Theorem 2.5.12]{D}. There are three subgroups $Inndiag(N)$, $\Phi$, and $\Gamma$ in
	$Aut(N)$ such that every two of them have the trivial intersection and  $$Aut(N)=Inndiag(N)\Phi\Gamma.$$
	Here $\Phi$ is the field automorphism group of $N$, $\Gamma$ is the graph automorphism group, and $Inndag(N)$ is the inner-diagonal automorphism group of~$N$. The subgroup  $Inndag(N)$ is  normal in $Aut(N)$ and contains $Inn(N)$ by \cite[Theorem 2.5.12]{D}. We have that $$\pi(\Gamma)\cup\pi(Outdiag(N))\subseteq\pi(N),$$ where $Outdiag(N)=Inndiag(N)/Inn(N)$. Moreover $[\Phi\Gamma, \Phi] = 1$. 
	
	If follows from the Schur--Zassenhaus theorem that $A$ is conjugate in $Aut(N)$ to a subgroup of $\Phi$ and we may assume that
	$A \leq \Phi$. Moreover, as $UA$ is a Frobenius group, we have
	$$U = [U, A] \leq [Aut(N), \Phi]\leq [\Phi\Gamma, \Phi]\,Inndiag(N)= Inndiag(N).$$
	Furthermore, $U \cap N = 1$ implies that $U$ is isomorphic to a subgroup of
	$Outdiag(N)$. In particular, $d = |Outdiag(N)|>1$. This means that $N$ is not a Suzuki group and $2,3\in\pi(N)$ by \cite{FT} and \cite[Chapter II, Corollary 7.3]{Gl}.
	
	Assume that $N$ is not isomorphic to $$\begin{array}{ll}
		PSL^+(n,q)=PSL(n,q)\cong A_{n-1}(q) & \text{ and} \\
		PSL^-(n,q)=PSU(n,q)\cong {^2A}_{n-1}(q). &
	\end{array}
	$$
	Then $d\leq 4$ by \cite[Theorem 2.5.12]{D}, and $|U|\leq 4$.  In this case, $A$ is a $\{2,3\}$-group since $A\leq Aut(U)$, which contradicts the fact that $(|A|,|G|)=1$.
	
	Thus, we may assume that $N=PSL^\varepsilon(n,q)$, where $\varepsilon\in\{+,-\}$.
	It follows from \cite[Theorem 2.5.12]{D}, that $Outdiag(N)$ is cyclic of order $d=(n,q-\varepsilon1)$ (in fact, $Inndiag(N)\cong PGL^\varepsilon(n,q)$ in this case). This means that the elementary abelian $p$-group $U$ is cyclic, $|U|=p$, and $p\leq d\leq n$. Now, take $r\in\pi(A)$. Then $$r\leq|A|\leq|Aut(U)|=p-1<n.$$ Moreover, $(|A|,|G|)=1$ implies  $(r,2q)=1$ and $r$ divides $q^{r-1}-1=q^{r-1}-(\varepsilon1)^{r-1}$. But this means that $r$ divides
	$$
	|N|=\frac{1}{d}q^{n(n-1)/2}\prod_{i=1}^n(q^i-(\varepsilon1)^i),
	$$ which contradicts
	the fact that $|A|$ and $|G|$ are coprime.
	This completes the proof. 
	\end{proof}
	
	Suppose that $A$ acts on $G$. Let $S_p(G, A)$ denote the set of all $A$-invariant Sylow $p$-subgroups of $G$, and $(O, A)_p(G)$ denote the intersection of all $P\in S_p(G,A).$
	
	\begin{corollaryn} Let $A$ act coprimely on  $G$ and let $p$ be a prime which is coprime to $|C_G(a)|$ for all $1\ne a\in A$. Then $(O, A)_p(G) = O_p(G)$ if the following hold:
		
		\textit{(i)  $G$ is $PSL(2,2^r)$ free for all $1\neq r$ dividing $|A|$ in case where $p \mid 2^r+1$,}
		
		\textit{(ii) $G$ is $Sz(2^r)$ free for all $1\neq r$ dividing $ |A|$ in case where $p \mid 4^r+1$.}
	\end{corollaryn}
	\begin{proof} We clearly have
		$O_p(G)\leq (O, A)_p(G)$. On the other hand, it is easy to see that $(O,A)_p(G)$ is
		normalized by $C_G(A)$. Since $p$ is coprime to $|C_G(a)|$ for all $1\ne a\in A$, we have $(O,A)_p(G)\leq O_p(G)$ by
		Theorem E as claimed.
	\end{proof}

	\section*{Acknowledgement} We thank Richard Lyons and Stephan Kohl whose remarks are directly applied to the proof of Lemma \ref{3 cases}.

\end{document}